\documentclass[times,doublespace,letterpaper]{oupau}
\usepackage[utf8]{inputenc}
\usepackage[english]{babel}
\usepackage{amssymb,amsmath,amsthm}

\newtheorem*{theorem_old}{Theorem}
\theoremstyle{definition}
\newtheorem{remark}{Remark}[section]

\DeclareMathOperator{\ord}{ord}

\DeclareMathOperator{\codim}{codim}
\DeclareMathOperator{\Ker}{Ker}
\DeclareMathOperator{\Imm}{Im}
\DeclareMathOperator{\dd}{d}
\DeclareMathOperator{\dt}{dt}
\DeclareMathOperator{\Spec}{Spec}
\DeclareMathOperator{\Hom}{Hom}

\begin{document}

\title{A differential analog of the Noether normalization lemma}

\author{Gleb Pogudin\affil{1}}
\abbrevauthor{G. Pogudin}
\headabbrevauthor{Pogudin, G.}

\address{\affilnum{1}Department of Mechanics and Mathematics, Moscow State University}
\correspdetails{pogudin.gleb@gmail.com}

\received{}
\communicated{}

\begin{abstract}
	In this paper, we prove the following differential analog of the Noether normalization lemma: for every $d$-dimensional differential algebraic variety over differentially closed field of zero characteristic there exists a surjective map onto the $d$-dimensional affine space.
    Equivalently, for every integral differential algebra $A$ over differential field of zero characteristic there exist differentially independent $b_1, \ldots, b_d$ such that $A$ is differentially algebraic over subalgebra $B$ differentially generated by $b_1, \ldots, b_d$, and whenever $\mathfrak{p} \subset B$ is a prime differential ideal, there exists a prime differential ideal $\mathfrak{q} \subset A$ such that $\mathfrak{p} = B \cap \mathfrak{q}$.
    
    We also prove the analogous theorem for differential algebraic varieties over the ring of formal power series over an algebraically closed differential field and present some applications to differential equations.
	\\
    \emph{MSC2010: 12H05.}
\end{abstract}

%%%%%%%%%%%%%%%%%%%%%%%%%%%%%%%

\maketitle

\section{Introduction}

	The Noether normalization lemma is an important result in commutative algebra and algebraic geometry.
    Stated geometrically, it says that for every affine algebraic variety of dimension~$d$ there exists a surjective map onto the $d$-dimensional affine space.
    One of possible (though not the strongest one) algebraic formulations is the following: for an integral algebra $A$ there exist algebraically independent $b_1, \ldots, b_d$ such that $A$ is algebraic over $B = k[b_1, \ldots, b_d]$, and whenever $\mathfrak{p} \subset B$ is a prime ideal, there exists a prime ideal $\mathfrak{q} \subset A$ such that $\mathfrak{p} = B \cap \mathfrak{q}$.
    More precisely, it turns out that the extension $B \subset A$ is integral.
    In this case due to the Cohen-Seidenberg theorem, the extension $A \subset B$ satisfies going-up and going-down properties for prime ideals.
    Besides its theoretical importance, the Noether normalization lemma is widely used in computational algebraic geometry (see \cite{mora}).

    As far as we know, there is no appropriate notion of integral extension for differential algebras.
    Going-up and going-down properties for quasi-prime differential ideals were investigated by Keigher in \cite{keigher}.
    Trushin extended the Cohen-Seidenberg theorem to the case of integral extensions (in the sense of commutative algebra) of differential algebras (\cite[th. 11-12]{trushin}).
    Geometrically, these extensions correspond to maps with finite fibers.
    However, in differential algebraic geometry zero-dimensional varieties are often infinite.
    In \cite[th. 1]{rosenfeld}, Rosenfeld showed that, for a differentially finitely generated extension $B \subset A$ of $k$-algebras without zero divisors, there exists an element $b \in B$ such that every prime ideal in a localization $B_b$ can be extended to a prime ideal in $A_b$ (see also paper by Kac \cite{kac}).
    
    We prove the following differential analogue of the Noether normalization lemma:
    \begin{itemize}
		\item For every $d$-dimensional differential algebraic variety over differentially closed field there exists a surjective map onto the $d$-dimensional affine space (Corollary~\ref{cor:alg_geom}).
        \item For every integral differential algebra $A$ there exist differentially independent $b_1, \ldots, b_d$ such that $A$ is differentially algebraic over subalgebra $B$ differentially generated by $b_1, \ldots, b_d$, and whenever $\mathfrak{p} \subset B$ is a prime differential ideal, there exists a prime differential ideal $\mathfrak{q} \subset A$ such that $\mathfrak{p} = B \cap \mathfrak{q}$ (Theorem~\ref{th:diff_alg}).
	\end{itemize}
    From the standpoint of differential equations Corollary~\ref{cor:alg_geom} means that for every system of differential equations in $n$ variables defining $d$-dimensional variety there exists an invertible change of variables such that for any tuple $(f_1, \ldots, f_d)$ of functions there exists a solution $(f_1, \ldots, f_n)$ extending this tuple.
    By function here we understand some element of the differentially closed field.
	However, we provide also an alternative version of Corollary~\ref{cor:alg_geom}, namely, Corollary~\ref{cor:infty}, which allows us to formulate similar result for solutions in a power series ring. See Section~\ref{sec:power_series} for details and Proposition~\ref{prop:complex_power_series} and Proposition~\ref{prop:complex_polynomial} for applications to algebraic differential equations over~$\mathbb{C}$ and over~$\mathbb{C}[t]$.
        
    There are two standard ways to prove the Noether normalization lemma: ''algebraic``, using some sufficiently general change of coordinates (see \cite[\S I.1]{mumford}), and  ''geometric``, using projections to hyperplanes of general position (\cite[\S II.8]{mumford}).
    Our proof is based upon the former.
    The main tool of the proof is a ``polynomial shift'' trick developed in the proof of \cite[th. 1]{me_primitive}.
    So, the proof is constructive in the sense that it gives an algorithm how to find a surjective map or elements $b_1, \ldots, b_d$.
    The second approach to algebraic Noether's normalization uses the completeness of projective varieties and the fact that the projective closure of an affine subvariety of the affine space does not contain all infinite points. 
    Interestingly, projective differential varieties are not complete (moreover, there are no complete differential algebraic varieties of positive dimension, see~\cite[Corollary 2.4]{pong}), and the projective closure of differential algebraic variety defined, for example, by $x^{\prime}y^{\prime} = 1$ contains the whole line at infinity. 
    
    The rest of the paper is organized as follows.
    In Section~\ref{sec:prelim} we introduce necessary notions from differential algebra and differential algebraic geometry.
    In Section~\ref{sec:nulstellensatz} we present specific differential $k$-algebras to be used in Theorem~\ref{th:alg_geom}.
    Section~\ref{sec:core_lemmas} contains several lemmas about differential polynomials, which are used further in the proof of Theorem~\ref{th:alg_geom}.
    Every lemma in Section~\ref{sec:core_lemmas} is supplied with references to the corresponding part of the proof of Theorem~\ref{th:alg_geom}.
    Section~\ref{sec:geometric} contains the proof of Theorem~\ref{th:alg_geom} (geometric version of the Noether normalization lemma) and its corollaries, namely, Corollary~\ref{cor:alg_geom} and Corollary~\ref{cor:infty}.
    In Section~\ref{th:diff_alg} we deduce Theorem~\ref{th:diff_alg} (algebraic version of the Noether normalization lemma) from Theorem~\ref{th:alg_geom}.
    In section~\ref{sec:power_series} we apply Corollary~\ref{cor:infty} to differential algebraic equations over~$\mathbb{C}$ and $\mathbb{C}[t]$  (Proposition~\ref{prop:complex_power_series} and Proposition~\ref{prop:complex_polynomial}).
    The last section contains some concluding remarks and acknowledgments.
	
%%%%%%%%%%%%%%%%%%%%%%%%%%%%%%%%%%%%%%%%%%%%%%%%%%%%%%%%

\section{Preliminaries}\label{sec:prelim}
    
    \subsection{Differential algebra}
    
    Throughout the paper all fields are assumed to be of characteristic zero.
	Let $R$ be a ring.
	A map $D\colon R \to R$ satisfying $D(a + b) = D(a) + D(b)$ and $D(ab) = aD(b) + D(a)b$ for all $a, b \in R$ is called a \textit{derivation}.
    A \textit{differential ring} $R$ is a ring with a specified derivation. 
	In this case we will denote $D(x)$ by $x^{\prime}$ and $D^n(x)$ by $x^{(n)}$.
	A differential ring which is a field will be called a \textit{differential field}.
	Let $F \subset E$ be a differential field extension and $a \in E$.
	Let us denote by $F\langle a \rangle$ the differential subfield of $E$ generated by $F$ and $a$.
    If $A$ is a ring without zero divisors, for $s \in A$ by $A_s$ we denote the localization of $A$ with respect to $s$.
    Furthermore, if $A$ is a differential ring, the derivation on $A$ can be uniquely extended to $A_s$.
    
    A differential ring $A$ is said to be a \textit{differential $k$-algebra} over a differential field $k$ if $A$ is a $k$-algebra and the restriction of the derivation of $A$ on $k$ coincides with the derivation on $k$.
    Let $A \subset B$ be a differential $k$-algebra extension and $b \in B$.
    Let us denote by $A\{ b\}$ the differential subalgebra of $B$ generated by $b$ and $A$.
	An ideal $I$ of a differential ring $R$ is said to be a \textit{differential ideal}, if $a^{\prime} \in I$ for all $a \in I$.
	The differential ideal generated by $a_1, \ldots, a_n \in I$ will be denoted by $[a_1, \ldots, a_n]$.
    A differential ideal $I$ is \textit{radical} if, whenever $a^n \in I$ for some $n > 0$, $a \in I$.
    The minimal radical differential ideal containing $a_1, \ldots, a_n$ will be denoted by $\{ a_1, \ldots, a_n \}$.

	Let $A$ be a differential $k$-algebra.
    We consider a polynomial ring $A[x, x^{\prime}, \ldots, x^{(n)}, \ldots]$, where $x, x^{\prime}, x^{\prime\prime}, \ldots$ are algebraically independent variables.
    Extending the derivation from $A$ to $A[x, x^{\prime}, \ldots]$ by $D(x^{(n)}) = x^{(n + 1)}$ we obtain a differential algebra.
    This algebra is called the \textit{algebra of differential polynomials} in $x$ over $A$ and we denote it by $A\{ x\}$.
    Iterating this construction, we define the algebra of differential polynomials in variables $x_1, \ldots, x_n$ over $A$ and denote it by $A\{x_1, \ldots, x_n\}$.
    
    Let $P \in k\{x_1, \ldots, x_n\}$ be a differential polynomial.
    The \textit{order} of $P$ with respect to $x_i$ is the largest $n$ such that $x_i^{(n)}$ occurs in $P$, or $-\infty$ if $P$ does not depend on $x_i$.
    We denote it by $\ord_{x_i} P$.
    If $s = \ord_{x_i} P \geqslant 0$ then we can define the \textit{separant} of $P$ with respect to $x_i$ as the partial derivative $\frac{\partial P}{\partial x_i^{(s)}}$.
    Moreover, we can write $P$ as a univariate polynomial in $x_i^{(s)}$: 
    $$I_0 \left( x_i^{(s)} \right)^{D} + I_1 \left( x_i^{(s)} \right)^{D - 1} + \ldots + I_D$$
    Then, $I_0$ is said to be the \textit{initial} of $P$ with respect to $x_i$.
    The importance of the separant is explained by the following lemma.
    
    \begin{lemma}\label{lem:partial_reduction}
		Let $P \in k\{ x_1, \ldots, x_n\}$, and assume that $\ord_{x_1} P \geqslant 0$.
        By $S$ we denote the separant of $P$ with respect to $x_1$.
        Then, for every differential polynomial $Q \in k\{ x_1, \ldots, x_n\}$ there exists a differential polynomial $\widetilde{Q}$ such that $\ord_{x_1} \widetilde{Q} \leqslant \ord_{x_1} P$ and
        $$
        S^NQ - \widetilde{Q} \in [P] \mbox{ for some } N \in \mathbb{Z}_{\geqslant 0}.
        $$
	\end{lemma}
    \begin{proof}
    	Let $\ord_{x_1} P = h$.
		The key ingredient of the proof is the fact that $P^{(N)}$ can be rewritten as $S x_1^{(N + h)} + T$, where $\ord_{x_1} T < h + N$.
        Assume that the order of $Q$ with respect to $x_1$ is $H > h$.
        By $D$ denote $\deg_{x_1^{(H)}} Q$, so 
        $$Q = C_0 \left(x_1^{(H)} \right)^D + C_1 \mbox{, where }\deg_{x_1^{(H)}} C_1 < D\mbox{ and }\ord_{x_1} C_0 < H$$
        Let $Q_1 = SQ - C_0 \left( x_1^{(H)} \right)^{D - 1} P^{(H - h)}$.
        Then, $\deg_{x_1^{(H)}} Q_1 < D$.
        
        Repeating this procedure we obtain a sequence of differential polynomials $Q = Q_0, Q_1, \ldots, Q_M$ such that $SQ_i - Q_{i + 1} \in [P]$ for all $i$ and $\ord_{x_1} Q_M \leqslant h$.
        Then, $S^MQ - Q_M \in [P]$, so we can set $\widetilde{Q} = Q_M$.
	\end{proof}
    
    \begin{remark}
		The above lemma and Lemma~\ref{lem:hypersurface} are special cases of the theory of characteristic sets of differential ideals.
        We do not need this machinery in full generality, so we state and prove these cases only.
        For more profound account see \cite{sit_dart}.
	\end{remark}
    
	Let $A \subset B$ be an extension of differential $k$-algebras without zero divisors.
    Elements $b_1, \ldots, b_n \in B$ are said to be \textit{differentially independent} over $A$ if $P(b_1, \ldots, b_n) \neq 0$ for all nonzero $P \in A\{x_1, \ldots, x_n\}$.
	An element $b \in B$ is said to be \textit{differentially algebraic over} $A$ if there exists a nonzero differential polynomial $P \in A\{x\}$ such that $P(b) = 0$.
    If all elements of $B$ are differentially algebraic over $A$, the extension $A \subset B$ is said to be \textit{differential algebraic}.
    A set $b_1, \ldots, b_n$ of differentially independent over $A$ elements is said to be a \textit{differential transcendence basis} if every $b \in B$ is differentially algebraic over $A\{b_1, \ldots, b_n\}$.
    Then, the number $n$ does not depend on the choice of basis (see \cite[p. 44]{ritt}) and is said to be the \textit{differential transcendence degree} of $B$ over $A$.
    If $A$ is not specified explicitly, we take it to be $k$.
    
    \begin{remark}\label{rem:resultant}
		Let us remind an important property of the resultant of univariate polynomials which we will use several times throughout the paper (for details see \cite[\S 3.1]{uag}).
        Let $P(x), Q(x) \in A[x]$, where $A$ is a commutative UFD.
        Then, the resultant of $P(x)$ and $Q(x)$ is an element of $A$ satisfying following properties:
        \begin{itemize}
			\item There exist $S_P(x), S_Q(x) \in A[x]$ such that $R = S_P(x) P(x) + S_Q(x) Q(x)$.
            
            \item If $P(x)$ and $Q(x)$ have no common divisors of positive degree in $A[x]$, then $R \neq 0$.
		\end{itemize}
        Thus, if $R \neq 0$ and $f \colon A[x] \to B$ is a ring homomorphism such that $f(P(x)) = 0$ and $f(R) \neq 0$, then $f(Q(x)) \neq 0$.
	\end{remark}

	\subsection{Differential algebraic geometry}
    \label{subsec:diff_alg_geom}
    
    Let us fix a differential field $k$ and a differential $k$-algebra $R$ without zero divisors.
    By $\mathbb{A}^n(R)$ we denote the affine $n$-dimensional space over $R$.
    Let $\Sigma$ be a subset in $k\{ x_1, \ldots, x_n\}$.
    By $V_R(\Sigma) = \{ p\in \mathbb{A}^n(R) \mid \forall S \in \Sigma \quad S(p) = 0 \}$ we denote a set \textit{$R$-points of a differential algebraic $k$-variety}.
    Equivalently, $V_R(\Sigma)$ is said to be a differential algebraic $R$-variety defined over $k$.
    The Ritt-Raudenbush theorem (see \cite[Theorem 7]{raudenbush}) implies that every differential algebraic variety can be defined by a finite $\Sigma$.
    Conversely, for a differential algebraic $R$-variety $X \subset \mathbb{A}^n(R)$ defined over $k$ let $I_k(X) = \{ S \in k\{x_1, \ldots, x_n\} \mid \forall p \in X \quad S(p) = 0 \}$.
    Clearly, $I_k(X)$ is a radical differential ideal.
    We call a differential $k$-algebra $k\{ x_1, \ldots, x_n\} / I_k(X)$ the \textit{$k$-coordinate ring} of $X$. 
    As in the case of affine algebraic geometry, there is one-to-one correspondence between $R$-points of $X$ and differential homomorphisms from the $k$-coordinate ring of $X$ to $R$.
    We can topologize $\mathbb{A}^n(R)$ by taking differential algebraic varieties as basic closed sets.
    This topology is referred to as the Kolchin topology.
    
    A differential algebraic $R$-variety $X$ defined over $k$ is said to be \textit{reducible} over $k$ if $X = X_1 \cup X_2$ for some differential algebraic $R$-varieties $X_1, X_2$ defined over $k$ such that $X_1 \subsetneq X$ and $X_2 \subsetneq X$.
    Otherwise, $X$ is \textit{irreducible} over $k$.
    $X$ is irreducible over $k$ if and only if its $k$-coordinate ring contains no zero divisors or, equivalently, $I_k(X)$ is a prime differential ideal.
	Every differential algebraic variety can be expressed as a finite union of irreducible varieties (see \cite[p. 22]{ritt}).
    
    We say that $\varphi\colon \mathbb{A}^n(R) \to \mathbb{A}^m(R)$ is a $k$-map if there exist differential polynomials $P_1, \ldots, P_m \in k\{ x_1, \ldots, x_n\}$ such that 
    $$\varphi\left( (a_1, \ldots, a_n) \right) = \left( P_1(a_1, \ldots, a_n), \ldots, P_m(a_1, \ldots, a_n)\right) \mbox{ for all } (a_1, \ldots, a_n) \in \mathbb{A}^n (R)$$
	A $k$-map $\varphi\colon \mathbb{A}^n(R) \to \mathbb{A}^m(R)$ defines a homomorphism of differential algebras 
    $$\varphi^{*} \colon k\{ y_1, \ldots, y_m\} \to k\{ x_1, \ldots, x_n\} \mbox{ by } \varphi^{*}(y_i) = P_i(x_1, \ldots, x_n) \mbox{ for all }i,$$
    and vice versa a homomorphism of differential algebras $f\colon k\{ y_1, \ldots, y_m\} \to k\{ x_1, \ldots, x_n\}$ defines a $k$-map 
    $$f^{\#}\colon \mathbb{A}^n(R) \to \mathbb{A}^m(R) \mbox{ by } f^{\#}(p) = (f(y_1)(p), \ldots, f(y_m)(p))$$
    
    For an irreducible over $k$ differential algebraic $R$-variety $X \subset \mathbb{A}^n(R)$ we define the dimension $\dim X$ as the differential transcendence degree of its $k$-coordinate ring $k\{x_1, \ldots, x_n\} / I_k(X)$.
    We set dimension of an arbitrary variety to be the maximum of the dimensions of its irreducible components.
    Let $\codim X = n - \dim X$.
    Irreducible varieties of codimension one admit the following characterization in terms of the corresponding ideal.
    \begin{lemma}\label{lem:hypersurface}
    	Let $I \subset k\{ x_1, \ldots, x_n \}$ be a prime differential ideal such that the differential transcendence degree of $k\{ x_1, \ldots, x_n \} / I$ is $n - 1$. 
		Then, there exists $i$ ($1 \leqslant i \leqslant n$) and irreducible differential polynomial $P \in k\{x_1, \ldots, x_n\}$ such that
        $$
        I = [P] : S^{\infty} = \{ Q \in K\{ x_1, \ldots, x_n\} \mid \exists n\colon S^nQ \in [P] \},
        $$
        where $S$ is the separant of $P$ with respect to $x_i$.
	\end{lemma}
    
    \begin{proof}
		Without loss of generality we may assume that $x_1, \ldots, x_{n - 1}$ are differentially independent modulo $I$.
		By $h$ denote the minimal possible order with respect to $x_n$ among all non-zero polynomials in $I$.
        Clearly, $h \geqslant 0$.
        Let $P \in I$ be a nonzero differential polynomial such that $\ord_{x_n} P = h$ and the total degree of $P$ is minimal possible.
        Since $I$ is prime, $P$ is irreducible.
        By $S$ denote the separant of $P$ with respect to $x_n$.
        
        Consider an arbitrary differential polynomial $Q \in I$.
        Lemma~\ref{lem:partial_reduction} implies that there exists $\widetilde{Q}$ such that $\ord_{x_n} \widetilde{Q} \leqslant h$ and $S^NQ - \widetilde{Q} \in [P] \subset I$.
        Then, $\widetilde{Q} \in I$.
        By $R$ denote the resultant of $\widetilde{Q}$ and $P$ with respect to $x_n^{(h)}$.
        Due to the properties of the resultant (see Remark~\ref{rem:resultant}), there exist $A, B \in k\{x_1, \ldots, x_{n - 1}\}[x_n, \ldots, x_n^{(h - 1)}]$ such that $R = AP + B\widetilde{Q}$.
        Hence, $R \in I \cap k\{x_1, \ldots, x_{n - 1}\}[x_n, \ldots, x_n^{(h - 1)}]$, so $R = 0$.
        Since $P$ is irreducible, $\widetilde{Q}$ is divisible by $P$, so $S^NQ \in [P]$ and $I \subset [P] : S^{\infty}$.
        
        On the other hand, consider $Q \in [P] : S^{\infty}$.
        There exists $N$ such that $S^NQ \in [P] \subset I$.
        Since $I$ is prime and $S \notin I$, $Q \in I$.
        So, $I \supset [P] : S^{\infty}$.
	\end{proof}
    
    %%%%%%%%%%%%%%%%%%%%%%%%%%%%%%%%%%%%%%%%%%%%%%%%%%%%%%%%%%%%%%%%%%%%%%%%%
    
    \section{Differential $k$-algebras with sufficiently many points}\label{sec:nulstellensatz}
    
    In Section~\ref{subsec:diff_alg_geom} we introduced basic notions of differential algebraic geometry in the case of points in an arbitrary differential $k$-domain.
    However, in what follows we are interested in differential domains such that there are ``sufficiently many points'' over them.
    
    \begin{definition}\label{def:many_points}
    	Let $R$ be a differential $k$-algebra over a differential field $k$.
        Assume that $R$ is a local ring (in a sense of commutative algebra) with algebraically closed residue field $k_0$.
        We will say that there are \emph{sufficiently many points} over $R$, if for every differentially finitely generated $k$-algebra $A$ the existence of a homomorphism $A \to k_0$ implies that there exists a differential homomorphism $A \to R$.
        We also note that there is an embedding of fields $k \to k_0$.
    \end{definition}
    
    \begin{proposition}\label{prop:many_nullstellensatz}
		Let $A$ be a differentially finitely generated $k$-algebra.
        Assume that there are sufficiently many points over $R$.
        Then, there exists a differential $k$-homomorphism $A \to R$.
	\end{proposition}
    
    \begin{proof}    
        Let $\mathfrak{m} \subset A$ be any maximal differential ideal.
        Since every maximal differential ideal over a field of zero characteristic is prime (see \cite[Proposition 1.19]{magid}), $A / \mathfrak{m}$ contains no zero divisors.
        Replacing $A$ with $A /\mathfrak{m}$ if necessary, we may assume that $A$ is integral.
        Due to Statement~5 from \cite{dima_inheritance} there exists an element $s \in A$ such that the localization $A_s$ can be written as $B[y_{\alpha}]$, where $B$ is a finitely generated $k$-algebra (not necessarily differential) and $\{ y_{\alpha}\}$ is an at most countable set of variables algebraically independent over $B$.
        Hilbert's Nullstellensatz implies that there exists a homomorphism $B \to k_0$.
        Moreover, it can be extended to a homomorphism $A_s = B[y_{\alpha}] \to k_0$ by sending elements of $\{ y_{\alpha} \}$ to arbitrary elements of $k_0$.
        Thus, we constructed a homomorphism $A \to k_0$.
        Then, a desired differential homomorphism exists due to Definition~\ref{def:many_points}.
    \end{proof}
    
    During the rest of the section we consider two types of such $k$-algebras, namely, differentially closed extensions and ring of power series (see Corollary~\ref{cor:nulstellensatz_diff_closed} and Corollary~\ref{cor:infty_many_points}).
    
    \subsection{Differentially closed fields}
    \label{subsec:diff_closed}

	A differential field $K$ is \textit{differentially closed} if, whenever $f, g \in K\{ x\}$, $g$ is nonzero and the order of $f$ is greater then the order of $g$, there is $a \in K$ such that $f(a) = 0$ and $g(a) \neq 0$.
	A differential field $K \supset k$ is a \textit{differential closure} of a differential field $k$ if $K$ is differentially closed field and the extension $k \subset K$ is differentially algebraic.
	Every differential field $k$ has a differential closure and any two differential closures of $k$ are isomorphic over $k$ (see \cite[Corollary 4.16]{scanlon_dart}).
	The importance of the differential closure can be illustrated by the following form of the differential Nullstellensatz (see \cite[Corollary 2.6]{marker}):
	\begin{theorem_old}
		If $k$ is a differential field and $\Sigma$ is a finite system of differential equations and inequations over $k$ such that $\Sigma$ has a solution in some $l \supset k$, then $\Sigma$ has a solution in any differentially closed $K \supset k$.
 	\end{theorem_old}
    
    \begin{corollary}\label{cor:nulstellensatz_diff_closed}
		Let $A$ be a finitely generated differential $k$-algebra, and $K \supset k$ is a differentially closed extension.
        Then, there exists a nonzero differential homomorphism $A \to K$.
	\end{corollary}
    
    Corollary~\ref{cor:nulstellensatz_diff_closed} implies that there are sufficiently many points over $K$.
    The following lemma shows that $k$-algebras with sufficiently many points behave similar to differential closures.

	\begin{lemma}\label{lem:nullstel_implies_diffclosed}
		Consider $f, g \in k\{ x \}$, where $g$ is nonzero, and the order of $f$ is greater then the order of $g$.
        Assume that there are sufficiently many points over $R$.
        Then, there exists $a \in R$ such that $f(a) = 0$ and $g(a)$ is invertible.
	\end{lemma}
    
    \begin{proof}
		Let $K$ be the differential closure of $k$.
        There exists $b \in K$ such that $f(b) = 0$ and $g(b) \neq 0$.
        Consider a homomorphism $\varphi\colon k\{x\} \to K$ defined by $\varphi(h) = h(b)$ for all $h(x) \in k\{ x\}$.
        By $A$ we denote a localization $\left( \Imm \varphi\right)_{\varphi(g)}$.
        Since $A$ is a differentially finitely generated $k$-algebra, there exists a nonzero $\psi\colon A \to R$.
        Let $a = \psi\left( \varphi(x) \right)$.
        Then, $f(a) = 0$ and $g(a)$ is invertible.
	\end{proof}

    %%%%%%%%%%
    
    \subsection{Power series}
	\label{subsec:power_series}

    Let $k$ be an algebraically closed differential field.
	By $k_0$ we denote the underlying pure field of $k$.
    An algebra of formal power series $k_{\infty} = k_0[[t]]$ is a differential ring with respect to derivation $\frac{\dd}{\dt}$.
    Moreover, a differential homomorphism $i \colon k \to k_{\infty}$ defined by $i(a) = \sum\limits_{j = 0}\frac{a^{(j)}}{j!} t^j$ turns $k_{\infty}$ into a differential $k$-algebra.
    Let us emphasize that the product of $a \in k$ and $\sum\limits_{j = 0}^{\infty} b_j t^j$ is not equal to $\sum\limits_{j = 0}^{\infty} ab_j t^j$, but to $i(a) \cdot \sum\limits_{j = 0}^{\infty} b_j t^j$.
    However, in the cases when $k$ consists solely of constants the multiplication of an element of $k_{\infty}$ by an element of $k$ is component-wise.
    
    Let us define also a homomorphism $\varepsilon\colon k_{\infty} \to k_0$ by formula $\varepsilon\left( \sum\limits_{j = 0}^{\infty} a_j t^j \right) = a_0$.
    Every differential $k$-homomorphism $\tilde{f} \colon A \to k_{\infty}$ defines a homomorphism $\varepsilon \circ \tilde{f} \colon A \to k_0$.
    The following lemma shows that the converse also holds.

	\begin{lemma}\label{lem:taylor}
    	Let $A$ be a differential $k$-algebra.
        Let $f \colon A \to k_0$ be a homomorphism.
        Then, there exists a differential $k$-homomorphism $\tilde{f} \colon A \to k_\infty$ such that $\varepsilon \circ \tilde{f} = f$.
	\end{lemma}
        
    \begin{proof}
		It is easy to verify that the formula $\tilde{f}(x) = \sum\limits_{j = 0} \frac{f\left( x^{(j)}\right)}{j!} t^j$ defines a desired homomorphism.
	\end{proof}

	\begin{corollary}\label{cor:infty_many_points}
		There are sufficiently many points over $k_{\infty}$.
	\end{corollary}

	\begin{remark}
		The above proof uses so-called Taylor homomorphism. For details, see \cite[\S 44.3]{YuPmonogr}.
	\end{remark}

    \begin{remark}
		Lemma~\ref{lem:taylor} can be interpreted in the context of category theory.
        Let us define a functor $F$ from the category $k_0\text{-}\mathbf{Alg}$ of $k_0$-algebras to the category $k\text{-}\mathbf{DiffAlg}$ of differential $k$-algebras sending an algebra $A$ to a differential algebra $A_{\infty} = A[[t]]$.
        For $\varphi \in \Hom_{k_0\textit{-}\mathbf{Alg}} (A, B)$ we define $F(\varphi)$ by $$F(\varphi) \left( \sum\limits_{j = 0}^{\infty} a_j t^j \right) = \sum\limits_{j = 0}^{\infty} \varphi(a_j) t^j$$
        
        Then, Lemma~\ref{lem:taylor} follows from the fact that $F$ is a left adjoint to the forgetful functor from $k\text{-}\mathbf{DiffAlg}$ to $k_0\text{-}\mathbf{Alg}$.
        Indeed, a homomorphism $A \to k_0$ yields a differential homomorphism $A_{\infty} \to k_{\infty}$, and $A$ can be embedded to $A_{\infty}$ as a differential $k$-algebra analogously to the embedding $i\colon k \to k_{\infty}$.
        Thus, we have a composition $A \to A_\infty \to k_\infty$ which gives us a desired differential homomorphism.
	\end{remark}

	\begin{remark}
		Unlike the case of differentially closed extension, Lemma~\ref{lem:nullstel_implies_diffclosed} is not longer true if $f \in k_{\infty} \{ x\}$.
        For example, polynomial $x^2 - t$ has no zeros in $\mathbb{C}[[t]]$.
	\end{remark}

	%%%%%%%%%%%%%%%%%%%%%%%%%%%%%%%%%%%%%%%%%%%%%%%%%%%%%%%%%%%%%%%%%%%%%%%%%

	\section{Core lemmas}\label{sec:core_lemmas}

	The following lemma is a slight modification of a well-known lemma (see, \cite[p.35]{ritt}).

	\begin{lemma}\label{lem:substitution}
      	Let $P(x) \in k\{ x\}$ be a nonzero differential polynomial and $\ord P \leqslant h$.
        Assume that there exists $t \in k$ such that $t^{\prime} = 1$.
        Then, there exists a polynomial $s(t) \in \mathbb{Q}[t]$ such that $\deg_t s \leqslant h$ and $P(s(t)) \neq 0$.
	\end{lemma}

    \begin{proof}
        Let us introduce algebraically independent variables $a_0, \ldots, a_h$.
        Consider an expression 
        $$p_m = a_h h^{\underline{m}} t^{h - m} + \ldots + a_m m! \mbox{ } (m = 0, \ldots, n) \mbox{, where } h^{\underline{m}} = h(h - 1)\ldots (h - m + 1)$$
        Since $k(a_0, \ldots, a_h) = k(p_0, \ldots, p_h)$, $p_0$, $\ldots$, $p_h$ are algebraically independent over $k$.
        Thus, substituting $x^{(m)} = p_m$ (for $m = 1, \ldots, h$) in $P(x)$, we obtain a nonzero polynomial from $k[a_0, \ldots, a_h]$.
        Let $(b_0, \ldots, b_h) \in \mathbb{Q}^h$ be a non-vanishing point of this polynomial.
    	Hence, if $s(t) = b_h t^h + \ldots + b_0$, then $P(s(t)) \neq 0$.
	\end{proof}
        
    Iterating this lemma, we obtain the following corollary.

	\begin{corollary}\label{cor:substitution}
		Let $P(x_1, \ldots, x_l) \in k\{ x_1, \ldots, x_l \}$ be a nonzero differential polynomial such that $\ord_{x_i} P \leqslant h$ for each $1 \leqslant i \leqslant l$.
        Assume that there exists $t \in k$ such that $t^{\prime} = 1$.
        Then, there exist polynomials $s_1(t), \ldots, s_l(t) \in \mathbb{Q}[t]$ such that $\deg s_i \leqslant h$ for each $1 \leqslant i \leqslant l$ and $P(s_1(t), \ldots, s_l(t)) \neq 0$.
	\end{corollary}
    
    The following lemma serves for mapping $d$-dimensional variety emmbedded into the $n$-dimensional affine space to a subset of the $d + 1$-dimensional affine space of certain form. For application, see Lemma~\ref{lem:project_to_d1_geom}.
    
    \begin{lemma}\label{lem:project_to_d1}
		Let $A$ be a differential $k$-algebra differentially generated by $a_1, \ldots, a_n \in A$.
        Assume that $A$ contains no zero divisors and the differential transcendence degree of $A$ is $d > 0$.
        By $p \colon k\{ x_1, \ldots, x_n\} \to A$ we denote the canonical projection defined by $p(x_i) = a_i$.
        
        Then, there exists an injective differential homomorphism $f\colon k\{ y_1, \ldots, y_{d + 1} \} \to k\{x_1, \ldots, x_n\}$ such that $f(y_i) \in \mathbb{Q}[x_1, \ldots, x_n]$ for all $i$ and $A_{p(f(Q))} = \left(\Imm (p\circ f) \right)_{p(f(Q))}$ for some $Q \in k\{ y_1, \ldots, y_{d + 1}\}$.  
	\end{lemma}

	\begin{proof}
		Without loss of generality we may assume that $a_1, \ldots, a_d$ constitute a differential transcendence basis of $A$.
        Let $E$ be the field of fractions of $A$ and let $F$ be the field of fractions of $k\{a_1, \ldots, a_d \} \subset A$.
        We want to apply Kolchin's primitive element theorem to the extension $F \subset E$.
        Actually, in his paper \cite{kolchin_primitive}, Kolchin  proved the following stronger statement:
        \begin{theorem_old}
			Let $F \subset E = F\langle a_1, \ldots, a_n\rangle$ be an extension of differential fields.
            Assume that $a_i$ is differentially algebraic over $F$ for all $i$.
            Let $b \in F$ be a nonconstant element, i.e. $b^{\prime} \neq 0$.
            Then, there exist polynomials $p_2(t), \ldots, p_n(t) \in \mathbb{Q}[t]$ such that:
            $$
            E = F\langle a_1 + p_2(b)a_2 + \ldots + p_n(b)a_n \rangle.
            $$
		\end{theorem_old}
        Thus, there exists $b_{d + 1} = a_{d + 1} + p_{d + 2}(a_1)a_{d + 2} + \ldots + p_n(a_1)a_n$, where $p_i(a_1) \in \mathbb{Q}\left[ a_1 \right]$, such that $E = F\langle b_{d + 1} \rangle$.
        Then, there exist differential polynomials $Q, T_{d + 1}, \ldots, T_{n} \in k\{ y_1, \ldots, y_{d + 1}\}$ such that for all $i > d$
        \begin{equation}\label{eq:project_to_d1}
        a_i = \frac{ T_i(a_1, \ldots, a_d, b_{d + 1}) }{ Q(a_1, \ldots, a_d, b_{d + 1}) }.
        \end{equation}
        Let $c = Q(a_1, \ldots, a_d, b_{d + 1})$.
        Then
        $$
        A_c = k\{ a_1, \ldots, a_d, b_{d + 1}\}_c
        $$
        So, we can set $f(y_i) = x_i$, $f(y_{d + 1}) = x_{d + 1} + p_{d + 2}(x_1)x_{d + 2} + \ldots + p_n(x_1)x_n$.
	\end{proof}

	\begin{remark}\label{rem:factorization}
		The resulting homomorphism can be factored as $f = a\circ i$, where $i \colon k\{ y_1, \ldots, y_{d + 1} \} \to k\{ y_1, \ldots, y_n\}$ is defined by $i(y_i) = y_i$ for all $i \leqslant d + 1$, and $a \colon k\{ y_1, \ldots, y_{n} \} \to k\{ x_1, \ldots, x_n\}$ is a differential isomorphism defined by $a(y_i) = f(y_i)$ for all $i \leqslant d + 1$ and $a(y_i) = x_i$ for all $i > d + 1$.
	\end{remark}

	%%%%%%%
    
    The following lemma replaces the subset of the $d + 1$-dimensional affine space produced by Lemma~\ref{lem:project_to_d1} with simpler one.
    For application, see Lemma~\ref{lem:two_polynomials_geom}.
    
    \begin{lemma}\label{lem:two_polynomials}
		Let $I \subset k\{ y_1, \ldots, y_{d + 1} \}$ be a nonzero prime differential ideal such that $y_1, \ldots, y_d$ are differentially independent modulo $I$.
        Let $Q$ be a differential polynomial such that $Q \notin I$.
        
        Then, there exist differential polynomials $P, S \in k\{ y_1, \ldots, y_{d + 1} \}$ such that
        \begin{enumerate}
			\item $\ord_{y_{d + 1}} P > \ord_{y_{d + 1}} S$.
            \item for every homomorphism (not necessarily differential) $f \colon k\{ y_1, \ldots, y_{d + 1} \} \to B$, where $B$ is a domain, if $[P] \subset \Ker f$ and $f(S) \neq 0$, then $I \subset \Ker f$ and $Q \notin \Ker f$.
		\end{enumerate}
	\end{lemma}
    
    \begin{proof}
		Lemma~\ref{lem:hypersurface} implies that $I = [P] : S_P^{\infty}$, where $P \in k\{y_1, \ldots, y_{d + 1}\}$ is irreducible and $S_P$ is the separant of $P$ with respect to $y_{d + 1}$.
        Lemma~\ref{lem:partial_reduction} implies that there exists $\widetilde{Q} \in k\{y_1, \ldots, y_{d + 1}\}$ such that $\ord_{y_{d + 1}} \widetilde{Q} \leqslant \ord_{ y_{d + 1} } P$ and $S_P^NQ - \widetilde{Q} \in [P]$ for some $N$.
        Since $Q\notin I$, $\widetilde{Q}$ is not zero, and $\widetilde{Q}$ is not divisible by $P$.
        
        By $h$ we denote the order of $P$ with respect to $y_{d + 1}$.
        Let $R$ be the resultant of $P$ and $S_P \widetilde{Q}$ with respect to $y_{d + 1}^{(h)}$.
        Then, $\ord_{y_{d + 1}} R < h$.
        Since $P$ is irreducible and neither $S_P$ nor $\widetilde{Q}$ is divisible by $P$, it follows that $R \neq 0$ (see Remark~\ref{rem:resultant}).
        
        Let $f \colon k\{ y_1, \ldots, y_{d + 1} \} \to B$ be an arbitrary homomorphism to a domain such that $f([P]) = 0$ and $f(R) \neq 0$.
        Then, $f(S_P) \neq 0$ and $f(\widetilde{Q}) \neq 0$ (see Remark~\ref{rem:resultant}).
        Let $T \in I$, then there exists $n$ such that $S_P^n T \in [P]$.
        Applying $f$, we obtain that $f(S_P)^n f(T) = 0$.
        Since there are no zero divisors in $B$ and $f(S_P) \neq 0$, $f(T) = 0$.
		Moreover, applying $f$ to $\widetilde{Q} - S_P^NQ \in [P]$ we obtain $f(Q) f(S_P)^N = f(\widetilde{Q}) \neq 0$, so $Q \notin \Ker f$.
	\end{proof}
    
    %%%%%%%

	The following two lemmas constitute the core of the proof of Theorem~\ref{th:alg_geom}.
    For application, see Lemma~\ref{lem:shifting_geom}.

    \begin{lemma}\label{lem:high_order}
		Let $P, S \in k\{ y_1, \ldots, y_{d + 1} \}$, $\ord_{y_{d + 1}} P > \ord_{y_{d + 1}} S$ and $P$ is irreducible.
        Then, there exists a differential automorphism $f_1 \colon k\{ y_1, \ldots, y_{d + 1} \} \to k\{ y_1, \ldots, y_{d + 1} \}$ such that the following inequalities hold:
        \begin{enumerate}
 			\item $\ord_{y_{d + 1}} f_1(P) > \ord_{y_{d + 1}} f_1(S)$;
            \item $\ord_{y_{d + 1}} f_1(P) > \max\left(\ord_{y_i} f_1(P), \ord_{y_i} f_1(S)\right)$ for all $i \leqslant d$.
		\end{enumerate}
	\end{lemma}
        
    \begin{proof}
		Let us choose a natural number $N > \max\left( \ord P, \ord S \right)$.
        We define $f_1$ by $f_1(y_{d + 1}) = y_{1} + y_{d + 1}^{(N)}$, $f_1(y_1) = y_{d + 1}$ and $f_1(y_i) = y_i$ for all $1 < i \leqslant d$.
            
        Let $h = \ord_{y_{d + 1}} P$ and $D = \deg_{y_{d + 1}^{(h)}} P$.
        Since $\left( y_{d + 1}^{(h + N)} \right)^D$ occurs in $f_1(P)$ only in $f_1\left( y_{d + 1}^h \right) = \left( y_{1}^{(h)} + y_{d + 1}^{(N + h)} \right)^D$, it can not cancel out.
        Hence, $\ord_{y_{d + 1}} f_1(P) = N + h$.
        Similarly, $\ord_{y_{d + 1}} f_1(S) = N + \ord_{y_{d + 1}} S$.
        Moreover, since $\max\left(\ord_{y_i} f_1(P), \ord_{y_i} f_1(S)\right) < N$ for all $i \leqslant d$, the desired inequalities hold.
	\end{proof}
 
 	\begin{definition}
		Let us consider a differential polynomial $Q(y_1, \ldots, y_n) \in k\{ y_1, \ldots, y_{d + 1}\}$ as a differential polynomial in $y_{d + 1}$ over $k\{ y_1, \ldots, y_{d} \}$.
		Then, $Q$ is said to be \textit{$y_{d + 1}$-manageable} if at least one coefficient lies in $k^{*}$.
	\end{definition}
        
    For example, $P = 2y_{d + 1}y_{d + 1}^{\prime\prime} - y_{d + 1}y_1 + \left( y_1^{\prime} \right)^2$ is $y_{d + 1}$-manageable differential polynomial, while $Q = 2y_{d + 1}y_{d + 1}^{\prime\prime} + y_{d + 1}y_{d + 1}^{\prime\prime}y_1 + \left( y_1^{\prime} \right)^2$ is not.
    
	\begin{lemma}\label{lem:manageable}
		Let $Q(y_1, \ldots, y_{d + 1}) \in k\{ y_1, \ldots, y_{d + 1}\}$.
		Then, there exists an automorphism $f_2$ of $k\{ y_1, \ldots, y_{d + 1}\}$ of the form $f_2(y_i) = y_i + p_i(y_{d + 1})$, where $p_i(t) \in \mathbb{Q}[t]$, for $i \leqslant d$ and $f_2(y_{d + 1}) = y_{d + 1}$ such that $f_2(Q)$ is $y_{d + 1}$-manageable.
	\end{lemma}
    
	\begin{proof}
    	The following proof uses the same ``differential shifting'' trick as the proof of Theorem 1 in \cite{me_primitive}.
		By $D_1$ denote the total degree of $Q$ with respect to $y_{d + 1}$ and its derivatives.
		By $D_2$ denote the total degree of $Q$ with respect to $y_1, \ldots, y_{d}$ and their derivatives.
        (for example, if $Q = y_{d + 1} y_{d + 1}^{\prime\prime}y_1^2 + y_{d + 1}y_1^{\prime}y_2^4$, then $D_1 = 2$ and $D_2 = 5$).
		Let $N = D_1 + D_2\ord Q + 1$.
		Let $V_{\Lambda} = \{\Lambda_{i, j} | i \geqslant 0, 1 \leqslant j \leqslant d\}$ be a set of algebraically independent variables.
		We extend the derivation from $k\{ y_1, \ldots, y_{d + 1}\}$ to $A = k\{ y_1, \ldots, y_{d + 1} \}\left[V_{\Lambda} \right]$ by $\left( \Lambda_{i, j} \right)^{\prime} = y_{d + 1}^{\prime}\Lambda_{i + 1, j}$.
		Let us construct $f_2$ as a composition of differential homomorphisms $k\{y_1, \ldots, y_{d + 1}\} \xrightarrow{i} A \xrightarrow{g} k\{ y_1, \ldots, y_{d + 1}\}$, where 
        \begin{itemize}
        	\item $i$ is a monomorphism defined by $i(y_j) = y_j + y_{d + 1}^N \Lambda_{0, j}$ for $j \leqslant d$ and $i(y_{d + 1}) = y_{d + 1}$;
            \item $g$ is defined by $g(y_i) = y_i$ and $g(\Lambda_{i, j}) = p_j^{(i)}(y_{d + 1})$, where $p_j(t) \in \mathbb{Q}[t]$ and $\deg p_j \leqslant \ord Q$ ($j = 1, \ldots, d$).
        \end{itemize}

		Let us rewrite $\widetilde{Q} = i(Q) \in A$ as a sum $\widetilde{Q}_{+} + \widetilde{Q}_{-}$, where $\widetilde{Q}_{+}$ includes all monomials of degree $D_2$ with respect to variables $\Lambda_{i, j}$ and $\widetilde{Q}_{-} = \widetilde{Q} - \widetilde{Q}_{+}$.
		Since $D_2$ is the maximal possible degree with respect to variables $\Lambda_{i, j}$, $\widetilde{Q}_{+}$ does not depend on $y_1, \ldots, y_{d}$.

		Let us show that $\widetilde{Q}_{+} \neq 0$.
		Consider $Q$ as a polynomial in $y_1, \ldots, y_{d}$ over $k\{ y_{d + 1} \}$.
		Among the monomials of maximal degree consider the monomial of maximal weight, say, $M = A(y_{d + 1}) y_{i_1}^{(k_1)}\cdot \ldots \cdot y_{i_l}^{(k_l)}$, where $A(y_{d + 1}) \in k\{ y_{d + 1} \}$.
		Then, $\widetilde{Q}_{+}$ includes the monomial 
        $$A(y_{d + 1})y_{d + 1}^{ND_2}\left( y_{d + 1}^{\prime}\right)^{k_1 + \ldots + k_l}\Lambda_{k_1, i_1}\cdot\ldots\cdot\Lambda_{k_l, i_l}$$ 
        which can not cancel out.

		The total degree of $g( \widetilde{Q}_{-} )$ with respect to $y_{d + 1}$ and its derivatives does not exceed $(D_2 - 1)(N + \ord Q) + D_1$.
		On the other hand, the total degree of every monomial of $g( \widetilde{Q}_{+} )$ is at least $ND_2$.
		The definition of $N$ implies that $ND_2 > D_1 + (D_2 - 1)(N + \ord Q)$.
        Hence, the set of monomials of $g( \widetilde{Q}_{+} )$ does not intersect with the set of monomials of $g( \widetilde{Q}_{-} )$ for all $g$.
        Since $g( \widetilde{Q}_{+} )$ does not involve $y_1, \ldots, y_d$, it suffices to find $p_j(t)$ ($j = 1, \ldots, d$) such that $g(\widetilde{Q}_{+}) \neq 0$.
		Let us define a derivation on the field of fractions of $A$ by $D(z) = \frac{z^{\prime}}{y_{d + 1}^{\prime}}$.
		By the definition, $D(y_{d + 1}) = 1$ and $D(\Lambda_{i, j}) = \Lambda_{i + 1, j}$.
		Then, due to Corollary~\ref{cor:substitution}, there exist $p_j$ ($j = 1, \ldots, d$) such that $g( \widetilde{Q}_{+} ) \neq 0$.
	\end{proof}

	%%%%%%%%%%%%%%%%%%%%%%%%%%%%%%%%%%%%%%%%%%%%%%%%%%%%%%%%%%%%%%%%%%%%%%%%%

	\section{Geometric version}\label{sec:geometric}

	\begin{theorem}\label{th:alg_geom}
    	Let $k$ be a differential field, and $R$ be a differential $k$-algebra with sufficiently many points.
		Let $X \subset \mathbb{A}^n(R)$ be a differential algebraic $R$-variety defined over $k$ of dimension $d > 0$.
        Then, there exists a $\mathbb{Q}$-map $\varphi \colon \mathbb{A}^n(R) \to \mathbb{A}^d(R)$ such that $\varphi(X) = \mathbb{A}^d(R)$.
	\end{theorem}

	\begin{corollary}\label{cor:alg_geom}
     	Let $K$ be a differentially closed field.
 		Let $X \subset \mathbb{A}^n(K)$ be a differential algebraic variety of dimension $d > 0$.
        Then, there exists a $\mathbb{Q}$-map $\varphi \colon \mathbb{A}^n(K) \to \mathbb{A}^d(K)$ such that $\varphi(X) = \mathbb{A}^d(K)$.
 	\end{corollary}

    \begin{corollary}\label{cor:infty}
    	Let $k$ be an algebraically closed field.
		Let $X \subset \mathbb{A}^n(k_{\infty})$ be a differential algebraic $k_{\infty}$-variety of dimension $d > 0$ defined over $k$.
        Then, there exists a $\mathbb{Q}$-map $\varphi \colon \mathbb{A}^n(k_{\infty}) \to \mathbb{A}^d(k_{\infty})$ such that $\varphi\left(X \right) = \mathbb{A}^d(k_{\infty})$.
	\end{corollary}

	%%%%%%%%%%%%%%%%

	\begin{proof}

		By $\varepsilon\colon R \to k_0$ we denote the residue map.
        Then, $\varepsilon(a) \neq 0$ for $a \in R$ if and only if $a$ is invertible.
		For a differential polynomial $Q \in k\{ x_1, \ldots, x_n \}$ we define $V_{\varepsilon} (Q)$ by
    	$$V_{\varepsilon} (Q) = \{ p \in \mathbb{A}^n(R) \mid \varepsilon\left( Q(p) \right) = 0 \} \subset \mathbb{A}^n(R)$$

		\underline{Step~1: Map to $\mathbb{A}^{d + 1}(R)$}.
        \begin{lemma}\label{lem:project_to_d1_geom}
			Let $X \subset \mathbb{A}^n (R)$ be an irreducible over $k$ differential algebraic $R$-variety of dimension $d > 0$.
	        Then, there exists a $\mathbb{Q}$-map $\varphi \colon \mathbb{A}^n(R) \to \mathbb{A}^{d + 1}(R)$ such that $\varphi(X)$ contains a nonempty set of the form $Y \backslash V_\varepsilon (Q)$, where $Q(y_1, \ldots, y_{d + 1}) \in k\{y_1, \ldots, y_{d + 1}\}$ and $Y \subset \mathbb{A}^{d + 1}(R)$ is an irreducible over $k$ differential algebraic $R$-variety of codimesion one.
		\end{lemma}
    
    	\begin{proof}
            By $A$ denote the $k$-coordinate ring of $X$.
        	Let $p \colon k\{ x_1, \ldots, x_n\} \to k\{x_1, \ldots, x_n\} / I_k(X) = A$ be the canonical projection.
        	Due to Lemma~\ref{lem:project_to_d1} there exists an injective differential homomorphism $f \colon k\{ y_1, \ldots, y_{d + 1} \} \to k\{ x_1, \ldots, x_n \}$ such that the corresponding map $f^{\#} \colon \mathbb{A}^n(R) \to \mathbb{A}^{d + 1}(R)$ is defined over $\mathbb{Q}$ and there exists $Q \in k\{ y_1, \ldots, y_{d + 1}\}$ such that $A_{p(f(Q))} = \left( \Imm (p \circ f) \right)_{p(f(Q))}$ 
        
	        Let $Y \subset \mathbb{A}^{d + 1} (R)$ be the closure of $f^{\#}(X)$ with respect to the Kolchin topology.
	        Then, the $k$-coordinate ring of $Y$ is $\Imm (p\circ f) \subset A$, so $Y$ is irreducible.
        	Since the $k$-coordinate rings of $Y$ and $A$ become isomorphic after localization with respect to $p(f(Q))$, $\dim Y = d$.
            Furthermore, for every point $q_0 \in Y$ such that $Q(q_0) \in R$ is invertible element of $R$ (equivalently, $\varepsilon\left(Q(q_0)\right) \neq 0$) the corresponding homomorphism to $R$ can be extended to a $k$-homomorphism $A_{p(f(Q))} \to R$.
            Thus, there exists a point $q_1 \in X$ such that $f^{\#}(q_1) = q_0$.
	        Hence, $f^{\#}(X) \supseteq Y \backslash V_{\varepsilon}(Q)$.
	        The set $Y \backslash V_{\varepsilon}(Q)$ is nonempty because there exists at least one differential $k$-homomorphism $A_{p(f(Q))} \to R$ (see Proposition~\ref{prop:many_nullstellensatz}).
		\end{proof}

		\underline{Step 2: Replace the set with simpler one}
        
        \begin{lemma}\label{lem:two_polynomials_geom}
			Let $Y \backslash V_{\varepsilon}(Q) \subset \mathbb{A}^{d + 1}(R)$ be a nonempty subset, where $Y$ is an irreducible over $k$ differential algebraic $R$-variety of dimension $d$, and $Q \in k\{y_1, \ldots, y_{d + 1} \}$ is a differential polynomial.
    	    Then, there exist differential polynomials $P, S \in k\{y_1, \ldots, y_{d + 1} \}$ such that set $V([P]) \backslash V_{\varepsilon}(S)$ is nonempty and $Y \backslash V_\varepsilon(Q) \supset V([P]) \backslash V_\varepsilon(S)$.
        
	        Moreover, $P$ and $S$ can be chosen such that $\ord_{y_i} P > \ord_{y_i} S$ for some $i$, and $P$ is irreducible.
		\end{lemma}
        
        \begin{proof}
			Without loss of generality we may assume that $y_1, \ldots, y_d$ are differentially independent modulo $I_k(Y)$.
    	    Applying Lemma~\ref{lem:two_polynomials} to ideal $I_k(Y)$ and differential polynomial $Q$, we obtain polynomials $P$ and $S$.
            Since $\ord_{y_{d + 1}} P > \ord_{y_{d + 1}} S$ and $P$ is irreducible, differential algebra $A = \left(k\{ y_1, \ldots, y_{d + 1} \} / {P} \right)_{S}$ is a differentially finitely generated differential $k$-algebra.
            Thus, Proposition~\ref{prop:many_nullstellensatz} implies that $V([P]) \backslash V_{\varepsilon}(S)$ is nonempty.
            
            Consider a point $p \in V([P]) \backslash V_{\varepsilon}(S)$ or, equivalently, a differential homomorphism $f\colon k\{y_1, \ldots, y_{d + 1} \} \to R$ such that $[P] \subset \Ker f$ and $f(S)$ is invertible.
            Then, Lemma~\ref{lem:two_polynomials} implies that $I_k(Y) \subset \Ker f$.
            Moreover, since $\varepsilon\circ f (S) \neq 0$ Lemma~\ref{lem:two_polynomials} implies that $\varepsilon\circ f(Q) \neq 0$.
            Thus, $V([P]) \backslash V_{\varepsilon}(S) \subset Y \backslash V_\varepsilon(Q)$.
		\end{proof}
        
        %%%%%%%%%%
        
        \underline{Step 3: Map onto $\mathbb{A}^d(R)$.}
    
	    \begin{lemma}\label{lem:shifting_geom}
			Let $P, S \in k\{ y_1, \ldots, y_{d + 1} \}$, $\ord_{y_{d + 1}} P > \ord_{y_{d + 1}} S$, and $P$ is irreducible.
        	Then, there exists a $\mathbb{Q}$-map $\varphi\colon \mathbb{A}^{d + 1}(R) \to \mathbb{A}^{d}(R)$ such that $\varphi\left(V([P]) \backslash V_{\varepsilon}(S) \right) = \mathbb{A}^d(R)$.
		\end{lemma}
    
    \begin{proof}
    	The map $\varphi$ will be constructed as a composition of maps corresponding to automorphisms from Lemma~\ref{lem:high_order} and Lemma~\ref{lem:manageable} and the projection onto the first $d$ coordinates.
        Applying the map $f_1^{\#}$ from Lemma~\ref{lem:high_order}, we may further assume, that $P$ and $S$ satisfy the inequalities from Lemma~\ref{lem:high_order}.
        By $I_P$ and $S_P$ denote the initial and the separant of $P$ with respect to $y_{d + 1}$, respectively.
        By $R$ denote the resultant of $S_P$ and $P$ with respect to $y_{d + 1}$.
        Let us apply Lemma~\ref{lem:manageable} to $Q = I_P R S$ and obtain the automorphism $f_2$.
	    By $h$ denote the order of $P$ with respect to $y_{d + 1}$ and rewrite $P$ as
    	$$
	    A_D \left( y_{d + 1}^{(h)} \right)^{D} + A_{D - 1}\left( y_{d + 1}^{(h)} \right)^{D - 1} + \ldots + A_0, \mbox{ where } A_i \in k\{ y_1, \ldots, y_{d + 1} \} \mbox{ and } \ord_{y_{d + 1}} A_i < h.
    	$$
        Note that due to the inequalities from Lemma~\ref{lem:high_order}
        $$
        f_2(P) = f_2(A_D) \left( y_{d + 1}^{(h)} \right)^{D} + f_2(A_{D - 1})\left( y_{d + 1}^{(h)} \right)^{D - 1} + \ldots + f_2(A_0), \mbox{ where } \ord_{y_{d + 1}} f_2(A_i) < h.
        $$
        Hence, $\ord_{y_{d + 1}} P = \ord_{y_{d + 1}} f_2(P)$, $\ord_{y_{d + 1}} S = \ord_{y_{d + 1}} f_2(S)$, and $f_2(I_P)$ and $f_2(S_P)$ are the initial and the separant of $f_2(P)$ with respect to $y_{d + 1}$, respectively.
	    Automorphism $f_2$ defines the $\mathbb{Q}$-map $f_2^{\#}\colon \mathbb{A}^{d + 1}(R) \to \mathbb{A}^{d + 1}(R)$.
	    Applying $f_2^{\#}$, we may further assume that $Q = I_P R S$ is $y_{d + 1}$-manageable.
                
	    Let $\varphi_3 \colon \mathbb{A}^{d + 1}(R) \to \mathbb{A}^d(R)$ be the projection on the first $d$ coordinates.
	    We claim that $\varphi_3\left( V([P]) \backslash V_{\varepsilon}(S) \right) = \mathbb{A}^d(R)$.
        Consider an arbitrary $R$-point of $\mathbb{A}^d(R)$, that is a differential homomorphism $g\colon k\{ y_1, \ldots, y_d\} \to R$.
        This homomorphism endows $R$ with the structure of a differential $k\{ y_1, \ldots, y_d \}$-algebra.
        By $B$ we denote $k\{y_1, \ldots, y_{d + 1} \}_S / [P]$.
        Then, we need to prove that there exists a differential homomorphism $\tilde{g} \colon B \to R$ such that the composition of $\tilde{g}$ with the embedding $k\{ y_d, \ldots, y_d\} \to B$ is $g$.
        It is sufficient to prove that there is a differential homomorphism
        $$
        A = B \otimes_{ k\{ y_1, \ldots, y_d\} } R \to R.
        $$
        Then, $g$ will be given as a composition of $B \to A$ with this homomorphism.
        The existence of the above differential homomorphism would follow from the existence of a homomorphism $A \to k_0$.
        The latter can be constructed as a tensor product of the residue map $R \to k_0$ and a homomorphism $\overline{g}\colon B \to k_0$ such that the composition $k\{ y_1, \ldots, y_d\} \to B$ with $\overline{g}$ is the residue of $g$.
        
        Thus, it is sufficient to extend the residue $\overline{g}\colon k\{ y_1, \ldots, y_d\} \to k_0$ to a homomorphism $\overline{g}\colon k\{ y_1, \ldots, y_{d + 1} \} \to k_0$ such that $\overline{g}(S) \neq 0$ and $\overline{g}\left( P^{(j)} \right) = 0$ for all $j$.
        Due to $y_{d + 1}$-manageability, $Q(\overline{g}(y_1), \ldots, \overline{g}(y_d), y_{d + 1})$ is a nonzero differential polynomial.
        Thus, $P(\overline{g}(y_1), \ldots, \overline{g}(y_d), y_{d + 1})$ is also a nonzero differential polynomial.
        Moreover, since $I_P(\overline{g}(y_1), \ldots, \overline{g}(y_d), y_{d + 1}) \neq 0$,
        $$
        h = \ord_{y_{d + 1}} P(\overline{g}(y_1), \ldots, \overline{g}(y_d), y_{d + 1}) = \ord_{y_{d + 1}} P(y_1, \ldots, y_{d + 1}) > \ord_{y_{d + 1}} Q(\overline{g}(y_1), \ldots, \overline{g}(y_{d}), y_{d + 1})
        $$
        Hence, we can choose $\overline{g}(y_{d + 1}), \ldots, \overline{g}(y_{d + 1}^{(h)})$ such that $\overline{g}(Q) \neq 0$ and $\overline{g}(P) = 0$.
        Since $\overline{g}(R) \neq 0$, $\overline{g}(S_P)$ is not zero (see Remark~\ref{rem:resultant}).
        Now assume that we have already defined $\overline{g}(y_{d + 1}), \ldots, \overline{g}\left( y_{d + 1}^{(h + N)} \right)$ such that $\overline{g}(P^{(j)}) = 0$ for all $j \leqslant N$.
        Let us recall that (see proof of Lemma~\ref{lem:partial_reduction})
        $$
        P^{(N + 1)} = S_P y_{d + 1}^{(N + h + 1)} + T, \mbox{ where } \ord_{y_{d + 1}} T \leqslant N + h 
        $$
        In order to ensure that $\overline{g}(P^{(N + 1)}) = 0$, we set $\overline{g}(y_{d + 1}^{(N + h + 1)}) = \frac{-\overline{g}(T)}{\overline{g}(S_P)}$.
        Then, proceeding by induction on $N$ we obtain a desired homomorphism $\overline{g}\colon k\{y_1, \ldots, y_{d + 1} \} \to k_0$.
        
        Thus, we proved that $\varphi_3\left( V([P]) \backslash V_{\varepsilon}(S) \right) = \mathbb{A}^d(R)$.
        Hence, the composition $\varphi = \varphi_3 \circ f_2^{\#} \circ f_1^{\#}$ is a desired map.
    \end{proof}
    
    Let us return to the proof of Theorem~\ref{th:alg_geom}.
    First, if $X$ is reducible, we replace $X$ with any of its irreducible components of maximal dimension.
    Then, we apply Lemma~\ref{lem:project_to_d1_geom} and map a subset of $X$ surjectively onto the nonempty set of the form $Y \backslash V_{\varepsilon}(Q)$.
    Due to Lemma~\ref{lem:two_polynomials_geom}, $Y \backslash V_{\varepsilon}(Q)$ contains a subset of the form $V([P]) \backslash V_\varepsilon(S)$, where $P$ is irreducible, $\ord_{y_i} P > \ord_{y_i} S$ for some $i$ and $P, S \in k\{ y_1, \ldots, y_{d + 1}\}$.
    Then, we use Lemma~\ref{lem:shifting_geom} in order to map the set $V([P]) \backslash V_{\varepsilon}(S)$ surjectively onto $\mathbb{A}^d(R)$.
    The composition of these two projections is a desired $\mathbb{Q}$-map.
    
    \end{proof}

	%%%%%%%%%%%%%%%%%%%%%%%%%%%%%%%%%%%%%%%%%%%%%%%%%%%%%%%%%%%%%%%%%%%%%%%%%%%%%%%%%%%%%%%%%%%%%%%
    
    \section{Algebraic version}\label{sec:algebraic}
    
	\begin{theorem}\label{th:diff_alg}
		Let $A$ be a finitely generated differential $k$-algebra without zero divisors.
        Assume that the differential transcendence degree of $A$ over $k$ is $d > 0$.
        Then, there exist differentially independent $a_1, \ldots, a_d \in A$ such that for every prime differential ideal $\mathfrak{p} \subset k\{a_1, \ldots, a_d\}$ there exists prime differential ideal $\mathfrak{q} \subset A$ such that $\mathfrak{q} \cap k\{ a_1, \ldots, a_n \} = \mathfrak{p}$.
	\end{theorem}

	\begin{proof}
    	Let $K$ be the differential closure of $k$.
        Let $b_1, \ldots, b_n$ be a set of generators of $A$.
        By $p \colon K\{ x_1, \ldots, x_n\} \to K \otimes_k A$ we denote the surjective differential homomorphism defined by $p (x_i) = 1 \otimes b_i$.
        Then, $K \otimes_k A$ is a coordinate ring of $X = V(\Ker\pi)$, and $\dim X = d$.
        Due to Corollary~\ref{cor:alg_geom}, there exists a $\mathbb{Q}$-map $\varphi\colon \mathbb{A}^{n}(K) \to \mathbb{A}^{d}(K)$ such that $\varphi(X) = \mathbb{A}^{d}(K)$.
        We identify the coordinate ring of $\mathbb{A}^{d}(K)$ with $K\{y_1, \ldots, y_d\}$
        Set $a_1 = p(\varphi^{*}(y_1))$, $\ldots$, $a_d = p(\varphi^{*}(y_d))$.
        Since $\varphi$ is a $\mathbb{Q}$-map $a_1, \ldots, a_n$ belong to  $1 \otimes A$ and can be identified with some elements of $A$.
        Let $B = k\{ a_1, \ldots, a_d\} \subset A$.
        Every differential homomorphism $B \to K$ can be extended to a differential homomorphism $K \otimes B \to K$.
        The surjectivity of $\varphi$ implies that every differential homomorphism $K \otimes B \to K$ can be extended to a differential homomorphism $K \otimes A \to K$.
        
        In general, extending prime differential ideal $\mathfrak{p}$ from $B$ to $A$ is equivalent to extending a differential homomorphism $B \to L$, where $L$ is some differentially closed extension of $k$, to a differential homomorphism $A \to L$.
        Indeed, $L$ can be chosen to be the differential closure of the field of fractions of $B / \mathfrak{p}$.
        Thus, all we need is the following lemma (see also \cite[\S I.4]{cassidy}):
        \begin{lemma}
			Let $X$ and $Y$ be a differential algebraic varieties defined by ideals $I \subset K\{ x_1, \ldots, x_n\}$ and $J \subset K\{ y_1, \ldots, y_m\}$, respectively.
            Let $\varphi\colon \mathbb{A}^n(K) \to \mathbb{A}^m(K)$ such that $\varphi(X) \supset Y$, and $L \supset K$ is a differentially closed extension of $K$.
            By $X_L$ and $Y_L$ we denote differential algebraic varieties over $L$ defined by the same equations as $X$ and $Y$.
            By $\varphi_L$ we denote a map $\mathbb{A}^n(L) \to \mathbb{A}^m (L)$ defined by the same formulas as $\varphi$. 
            
            Then, $\varphi_L(X_L) \supset Y_L$.
		\end{lemma}
        
        \begin{proof}
        	By the Ritt-Raudenbash theorem every differential algebraic variety can be defined by a finite number of equations.
            By $P_1, \ldots, P_s$ and $Q_1, \ldots, Q_l$ we denote defining equations for $X$ and $Y$, respectively.
			The inclusion $\varphi(X) \supset Y$ can be expressed by the following first-order formula over $K$:
            \begin{multline*}
            \forall a_1, \ldots, a_m \left( Q_1(a_1, \ldots, a_m) = 0 \wedge \ldots \wedge Q_l(a_1, \ldots, a_m)\right) \to \\ \exists b_1, \ldots, b_n \left( P_1(b_1, \ldots, b_n) \wedge \ldots \wedge P_s(b_1, \ldots, b_n) \wedge \varphi(b_1, \ldots, b_m) = (a_1, \ldots, a_n)\right)
            \end{multline*}
            Due to model completness of differentially closed fields (see \cite[Corollary 2.5]{marker}), the same first-order formula holds over $L$.
            Thus, $\varphi_L(X_L) \supset Y_L$.
		\end{proof}
        
    \end{proof}

    %%%%%%%%%%%%%%%%%%%%%%%%%%%%%%%%%%%%%%%%%%%%%%%%%%%%%%%%%%%%%%%%%%%%%%%%%%%%%%%%%%%%%%%%%%%%%%%%
    
    \section{Applications to power series}\label{sec:power_series}
        
	\begin{proposition}\label{prop:complex_power_series}
    	Let $P_1(x_1, \ldots, x_n), \ldots, P_m(x_1, \ldots, x_n) \in \mathbb{C}\{x_1, \ldots, x_n\}$ be a set of differential polynomials such that radical differential ideal $\{P_1, \ldots, P_m\}$ is prime and the differential transcendence degree of $\mathbb{C}\{x_1, \ldots, x_n\} / \{P_1, \ldots, P_m\}$ equals $d > 0$.
        
        Then there exists an invertible change of variables $x_i = Q_i(y_1, \ldots, y_n) \in \mathbb{Q}\{y_1, \ldots, y_n\}$ ($i = 1, \ldots, n$) such that for every formal power series $f_1(t), \ldots, f_d(t) \in \mathbb{C} [[t]]$ the system
        $$
        \begin{cases}
			\widetilde{P}_1(f_1(t), \ldots, f_d(t), y_{d + 1}, \ldots, y_n) = 0 \\
            \vdots \\
            \widetilde{P}_m(f_1(t), \ldots, f_d(t), y_{d + 1}, \ldots, y_n) = 0 
		\end{cases},
        $$
        where $\widetilde{P}_i(y_1, \ldots, y_n) = P_i\left( Q_1(y_1, \ldots, y_n), \ldots, Q_n(y_1,\ldots, y_n ) \right)$, has a solution in the ring of formal power series $\mathbb{C}[[t]]$.
        
        Moreover, if power series $f_1(t), \ldots, f_d(t)$ converge in some neighbourhood of $t = 0$, then the above system has an analytic solution at $t = 0$.
    \end{proposition}

	\begin{proof}
		Both maps in the proof of Theorem~\ref{th:alg_geom} (i.e. the map from Lemma~\ref{lem:project_to_d1_geom} and the map from Lemma~\ref{lem:shifting_geom}) are compositions of an automorphism and a projection (see Remark~\ref{rem:factorization} and proofs of Lemma~\ref{lem:two_polynomials_geom} and Lemma~\ref{lem:shifting_geom}).
        Thus, the resulting map $\varphi\colon \mathbb{A}^n(k_{\infty}) \to \mathbb{A}^d(k_\infty)$ can be factored as a composition $\varphi = \pi \circ \alpha$, where $\alpha \colon \mathbb{A}^n(k_\infty) \to \mathbb{A}^n(k_\infty)$ is an automorphism, and $\pi\colon \mathbb{A}^n(k_\infty) \to \mathbb{A}^d(k_\infty)$ is a projection onto the first $d$ coordinates.
        In our case $k_\infty = \mathbb{C}[[t]]$, and an automorphism $\alpha^{*}\colon \mathbb{C}\{ x_1, \ldots, x_n\} \to \mathbb{C}\{ x_1, \ldots, x_n\}$ of differential $\mathbb{C}$-algebra gives us the desired change of variables by $y_i = \alpha^{*}(x_i)$ for all $i$. 
        
        In order to prove the last claim of the proposition let us analyze the end of the proof of Theorem~\ref{th:alg_geom}.
        Power series $f_{d + 1}(t)$ is constructed as a solution of the differential equation $P(f_1, \ldots, f_d, y_{d + 1}) = 0$ of order $h$, where $f_1, \ldots, f_d$ are analytic at $t = 0$.
        Moreover, the construction implies, that $\frac{\partial}{\partial y_{d + 1}^{(h)}} P(f_1, \ldots, f_{d + 1})(0) \neq 0$.
        Then, due to the Malgrange theorem (see \cite{malgrange}, \cite[Th. 1]{gontsov}) power series $f_{d + 1}(t)$ converges in some neighbourhood of $t = 0$.
        Since $f_i(t)$ are rational functions in $f_1, \ldots, f_{d + 1}$ and their derivatives for all $i > d + 1$, they are analytic at $t = 0$, too.
	\end{proof}

	\begin{remark}
		In practice ideal $\{ P_1, \ldots, P_m\}$ need not be prime, but there is an algorithm computing its decomposition as an intersection of prime differential ideals (see \cite[Section V.5]{ritt}), which corresponds to a decomposition of a differential algebraic variety as a finite union of irreducible varieties.
	\end{remark}

	\begin{proposition}\label{prop:complex_polynomial}
		Let $P_1(t, x_1, \ldots, x_n), \ldots, P_m(t, x_1, \ldots, x_n) \in \mathbb{C}[t]\{x_1, \ldots, x_n\}$ be a set of differential polynomials over the differential ring $\mathbb{C}[t]$ (where $t^{\prime} = 1$) such that radical differential ideal $\{P_1, \ldots, P_m\}$ is prime and the differential transcendence degree of $\mathbb{C}[t]\{x_1, \ldots, x_n\} / \{P_1, \ldots, P_m\}$ equals $d > 0$.
        
        Then there exists an invertible change of variables $x_i = Q_i(t, y_1, \ldots, y_n) \in \mathbb{Q}[t]\{y_1, \ldots, y_n\}$ ($i = 1, \ldots, n$) such that for every formal power series $f_1(s), \ldots, f_d(s) \in \mathbb{C} [[s]]$ there exists a constant $\lambda \in \mathbb{C}$ such that the system
        $$
        \begin{cases}
			\widetilde{P}_1(t, f_1(t - \lambda), \ldots, f_d(t-\lambda), y_{d + 1}, \ldots, y_n) = 0 \\
            \vdots \\
            \widetilde{P}_m(t, f_1(t - \lambda), \ldots, f_d(t - \lambda), y_{d + 1}, \ldots, y_n) = 0 
		\end{cases},
        $$
        where $\widetilde{P}_i(t, y_1, \ldots, y_n) = P_i\left( t, Q_1(t, y_1, \ldots, y_n), \ldots, Q_n(t, y_1,\ldots, y_n ) \right)$, has a solution in the ring of formal power series $\mathbb{C}[[t - \lambda]]$.
	\end{proposition}
    
    \begin{proof}
		Let us consider $t$ as one more differential indeterminate and add one more equation $t^{\prime} = 1$.
        An ideal remains prime, and the differential transcendence degree is the same.
        Like in the proof of Proposition~\ref{prop:complex_power_series}, Theorem~\ref{th:alg_geom} gives us an invertible change of coordinates $x_i = Q_i(y_1, \ldots, y_{n + 1})$, $t = Q_{n + 1}(y_1, \ldots, y_{n + 1})$.
        We claim that we can assume that $t = y_{n + 1}$.
        In order to prove the claim, let us analyze the proof of Lemma~\ref{lem:project_to_d1}.
        Since $t$ is differentially algebraic over $\mathbb{Q}$, $t$ can not be contained in a differential transcendence basis of $A$.
        Furthermore, we can assume that $t = a_{k}$, where $k > d + 1$ (we follow the notation of Lemma~\ref{lem:project_to_d1}).
        Thus, $t$ is not affected by any of automorphisms from the proof of Theorem~\ref{th:alg_geom}.
        Thus, we may assume that $t = y_{n + 1}$, after the change of variables we obtain a system of the form
        $$
        \begin{cases}
			\widetilde{P}_1(t, y_1, \ldots, y_n) = 0 \\
            \vdots \\
            \widetilde{P}_m(t, y_1, \ldots, y_n) = 0 \\
            t^{\prime} = 1
		\end{cases}
        $$
        
        Let us consider a partial substitution $y_1 = f_1(s), \ldots, y_d = f_d(s)$, where $f_i(s) \in \mathbb{C}[[s]]$.
        It can be extended to the solution in $\mathbb{C}[[s]]$ of the system above, namely $y_{d + 1} = f_{d + 1}(s), \ldots, y_n = f_n(s), t = f_{n + 1}(s)$.
        Since $t^{\prime} = 1$, $f_{n + 1}(s)$ is of the form $s + \lambda$ for some $\lambda \in \mathbb{C}$.
        Replacing $s$ with $t - \lambda$ everywhere, we obtain a desired solution.
	\end{proof}

%%%%%%%%%%%%%%%%%%%%%%%%%%%%

\section{Concluding remarks and acknowledgments}
    
    \begin{remark}
		Let us note that proofs in the paper could be organized in a significantly different way.
        Corollary~\ref{cor:infty} implies a nondifferential analogue of Theorem~\ref{th:diff_alg}, that is for every finitely generated integral differential algebra $A$ over a differential field $k$ there exist $b_1, \ldots, b_d \in A$ such that $A$ is differentially algebraic over $B = k\{ b_1, \ldots, b_d\}$, and for every prime (not necessarily differential) ideal $\mathfrak{p} \subset B$ there exists prime ideal $\mathfrak{q} \subset A$ such that $\mathfrak{p} = \mathfrak{q} \cap B$.
        Geometrically speaking, it means that there is a surjective map $\Spec A \to \Spec B$.
        Statement~2 from \cite{dima_inheritance} implies that the same map induces a surjective map from the differential spectra of $A$ onto the differential spectra of $B$.
        Thus, Theorem~\ref{th:diff_alg} follows from Corollary~\ref{cor:infty}.
        Furthermore, Corollary~\ref{th:alg_geom} can be deduced from Theorem~\ref{th:diff_alg} using the differential Nullstellensatz.
        
        In particular it means, that the $\mathbb{Q}$-map given by Theorem~\ref{th:alg_geom} for a differential algebraic $k_{\infty}$-variety defined over $k$ will be also surjective for $K$-variety defined by the same equations, where $K$ is the algebraic closure of $k$.
	\end{remark}

	\begin{remark}
		The following example shows that a surjective map onto the affine space for a differential algebraic $K$-variety defined over $k$, where $K$ is the differential closure of $k$, might not be surjective for $k_{\infty}$-variety defined by the same equations.
        
        Let $P = xy^{\prime} + (x^{\prime} + 1)y - 1 \in \mathbb{C}\{x, y\}$, and by $K$ denote the differential closure of $\mathbb{C}$.
        By $X \subset \mathbb{A}^2(K)$ denote an irreducible differential algebraic variety defined by $I = [P] : x^{\infty}$.
        Consider a projection $\pi\colon \mathbb{A}^2(K) \to \mathbb{A}^1(K)$ defined by $\pi((x_0, y_0)) = x_0$.
        
        We claim that $\pi(X) = \mathbb{A}^1(K)$.
        Choose an arbitrary $x_0 \in \mathbb{A}^1(K)$.
        If $x_0 \neq 0$, then there exists $y_0 \in K$ satisfying the equation $x_0y^{\prime} + (x_0^{\prime} + 1)y - 1 = 0$.
        Then, $(x_0, y_0) \in X$.
        Now it suffices to show that $(0, 1) \in X$.
        For an arbitrary differential polynomial $Q$ by $Q_0$ we denote a sum of monomials involving either $x^{(i)}$ for some $i$, or $y^{(j)}$ for some $j > 0$.
        Let $Q_1 = Q - Q_0 \in K[y]$.
        If that $Q \in I$, then there exists $N$ such that $x^NQ \in [P]$.
		Note that polynomials of the form $TP^{(i)}$, where $i > 0$, do not have common monomials with $x^NQ_1$, so $x^NQ_1$ is divisible by $y - 1$.
        Thus, $Q$ vanishes at point $(0, 1)$.
        
        On the other hand, there is no vanishing point of $P$ of the form $(-t, y) \in \mathbb{C}[[t]] = \mathbb{C}_{\infty}$, so the image of $X$ under projection $\mathbb{A}^2\left( \mathbb{C}[[t]] \right) \to \mathbb{A}^1\left( \mathbb{C}[[t]] \right)$ defined by the same formula as $\pi$ is not the whole affine line.
	\end{remark}

    The author is grateful to Dmitry Trushin, Yu.P. Razmyslow, Alexey Ovchinnikov, and Renat Gontsov for useful discussions and to the referees for their invaluable comments and suggestions.

%%%%%%%%%%%%%%%%%%%%%%%%%%%%

\end{document}